\documentclass[a4paper,11pt]{article}
\usepackage[utf8]{inputenc}
\usepackage[T1]{fontenc}

\usepackage{amsthm}
\usepackage{tikz}
\usepackage{mathrsfs,amssymb}  
\usepackage{enumitem}
\usepackage{fullpage}
\usetikzlibrary{arrows}
\usetikzlibrary{arrows.meta}

\usepackage{mathtools}
\usepackage{amssymb}
\usepackage[nomath]{lmodern}
\usepackage{graphicx}
\usepackage{pgf,tikz,tkz-graph,subcaption}
\usetikzlibrary{arrows,shapes}
\usetikzlibrary{decorations.pathreplacing}
\usepackage{tkz-berge}
\usepackage{enumitem}
\usepackage[normalem]{ulem}
\usepackage{hyperref}
\usepackage{comment}
\hypersetup{colorlinks = true, linkcolor = blue, citecolor = blue, urlcolor = blue}
\usepackage[capitalise]{cleveref}

\allowdisplaybreaks

\usepackage[margin=1in]{geometry}
\usepackage[textsize=small]{todonotes}

\usepackage{bookmark}

\newtheorem{theorem}{Theorem}
\newtheorem*{theorem*}{Theorem}

\newtheorem{lemma}[theorem]{Lemma}
\newtheorem{remark}[theorem]{Remark}

\newtheorem{proposition}[theorem]{Proposition}
\newtheorem{definition}[theorem]{Definition}

\newcommand*{\myproofname}{Proof}

\DeclareMathOperator{\W}{W}
\DeclareMathOperator{\ecc}{ecc}
\DeclareMathOperator{\diam}{diam}
\DeclareMathOperator{\rad}{rad}

\DeclareMathOperator{\tsp}{tsp}
\DeclareMathOperator{\Wtsp3}{W_{tsp,3}}
\DeclareMathOperator{\Wtspk}{W_{tsp,k}}
\DeclareMathOperator{\mutspk}{\mu_{tsp,k}}
\DeclareMathOperator{\mutsp3}{\mu_{tsp,3}}
\newcommand{\Exp}{\,\mathbb{E}}
\def \d{\, \mathrm{d}}
\def\vc{\overrightarrow}

\newcommand*{\floorfrac}[2]{\mathopen{}\left\lfloor\frac{#1}{#2}\right\rfloor\mathclose{}}

\newcommand*{\abs}[1]{\lvert #1\rvert}

\title{The average solution of a TSP instance in a graph}  
\author{
Stijn Cambie \thanks{Department of Computer Science, KU Leuven Campus Kulak-Kortrijk, 8500 Kortrijk, Belgium.\\ Supported by the Institute for Basic Science (IBS-R029-C4) and a FWO grant with grant number 1225224N.
E-mail: {\tt stijn.cambie@hotmail.com}.
}
}

\begin{document}
\parindent=0cm

\maketitle

\begin{abstract}
    We define the average $k$-TSP distance $\mutspk$ of a graph $G$ as the average length of a shortest closed walk visiting $k$ vertices, i.e. the expected length of the solution for a random TSP instance with $k$ uniformly random chosen vertices. We prove relations with the average $k$-Steiner distance and characterize the cases where equality occurs. We also give sharp bounds for $\mutspk(G)$ given the order of the graph.
\end{abstract}

\textbf{keywords: travelling salesman problem, average distance, Wiener index, Steiner distance, TSP distance}

\section{Introduction}\label{sec:intro}
The travelling salesman problem (TSP for short) asks for a shortest closed walk (for definitions and notation, the reader less familiar with the topic can first have a look at Subsection~\ref{subsec:def&not}) in a graph that contains a given set of vertices. This can be considered more generally for (weighted) digraphs.
It is an important problem in operations research and computer science. For more information, see e.g.~\cite{GP02}. 

The TSP problem has been studied in many directions, but a pure mathematical study of the extremal behaviour of the problem seems lacking.
That is, usually one solves the TSP for a given instance (a fixed weighted graph and vertices to visit) in e.g. logistics, manufacturing, or circuit board testing. 
One can also wonder in advance how the city plan, the location of every tool, or how the components on the circuit board have to be put such that the TSPs that have to be solved afterward are expected to have good solutions. 
As a more concrete example, in a hospital setting, a nurse can efficiently solve an instance of the Traveling Salesman Problem (TSP) to ensure the timely delivery of medication to all patients in need. When the hospital still has to be constructed, one can try to construct the hospital with the desired number of rooms and under practical conditions in such a way that the expected travel distance of the nurse will be small.
Here we can assume that patients who need some type of medication are uniformly distributed among rooms over time.

In this paper, we start this study in the most fundamental setting, and do so by relating the extremal properties on the average solution of certain TSP instances on graphs, with well-studied parameters such as the Steiner distance and total distance of a graph. These connections may help address some of the broader questions by linking them to established results or providing insight into the anticipated outcomes.
 
We define the analogue of Steiner distance and Steiner Wiener index for the TSP problem; the average $k$-TSP distance and $k$-TSP Wiener index.
The analogues of certain other distance measures such as e.g. eccentricity, diameter and radius for TSP can be defined as well, which is done in the conclusion for completeness.

\begin{definition}\label{def:Wtspk}
    In a graph $G$, for $k$ vertices $v_1,v_2, \ldots, v_k$ we define the $k$-TSP distance\\ $\tsp_k(v_1,v_2, \ldots, v_k)$ as the length of a shortest closed walk that visits all vertices $v_1,v_2, \ldots, v_k$.\\
    The $k$-TSP Wiener index of a graph $G$ is defined as 
    $$\Wtspk(G)= \sum_{\binom{V}{k}} \tsp_k(v_1,v_2, \ldots, v_k).$$\\
    The average $k$-TSP distance equals 
    $$\mutspk(G)=\frac{\Wtspk(G)}{\binom{n}{k}}.$$
\end{definition}

The $k$-TSP Wiener index of a digraph can be defined compeletely similar as in the graph case (see~\cref{def:Wtspk_Di}).
In both cases, one can consider weighted versions, as is the case for TSP instances as well.
In that case, every edge or arc is assigned some weight, and $\tsp_k(v_1,\ldots,v_k)$ denotes the minimum sum of weights over all closed walks visiting each of these vertices.
An asymmetric TSP has to be formulated with weights on the arcs of a digraph, where the arcs $\vc{uv}$ and $\vc{vu}$ do not necessarily get the same weight. 

Both the TSP Wiener index and Steiner Wiener index are proportional (with factors equal to $2$ and $1$ respectively) with the usual Wiener index when $k=2$.
When $k=3$, it was proven in~\cite[cor.~4.5]{LMG16} that for trees $W_3(T)=\frac{n-2}{2} W(T)$. Nevertheless this is not true for all graphs, while we will note that $\Wtsp3(G)=(n-2)W(G)$ is true for all graphs $G$ and so in some sense this variant behaves nicer.
As a main result, we characterize the graphs for which certain TSP Wiener indices and Steiner Wiener indices are equal up to exactly a factor $2.$
By dividing both quantities by $\binom{n}{k},$ one obtains the equivalent inequality $\mutspk(G)\le 2\mu_k(G).$

\begin{theorem}\label{thr:TSPW=2W}
    For a graph $G$ and integer $k\ge 2$, we have $\Wtspk(G)\le 2\W_k(G).$
    Equality holds if and only if $k=2$, $k=3$ and $G$ contains no vertices $u,v,w$ for which $2\max\{d(u,v),d(u,w),d(v,w)\}<d(u,v)+d(u,w)+d(v,w)$ and every choice of $3$ shortest paths between the $3$ pairs $(u,v),(v,w), (w,u)$ are edge-disjoint, or $k \ge 4$ and $G$ is a tree.
\end{theorem}

This is proven in Section~\ref{sec:rel_Wk_Wtspk}.
In Section~\ref{sec:bounds_mutspk} we prove sharp lower and upper bounds for the average $k$-TSP distance, and characterize the extremal graphs.

\begin{theorem}
    For every $k \ge 2$ and graph $G$ of order $n$, $$k \le \mutspk(G) \le 2\frac{k-1}{k+1}(n+1),$$
\end{theorem}

Finally, we also consider the analogue of the DeLaVi\~{n}a-Waller~\cite{DLVW} conjecture, which states that the cycle $C_{2d+1}$ attains the maximum average distance among all graphs with diameter $d$ and order $n=2d+1,$ for every $d \ge 3.$
We prove that this extension is false for $k \ge 4$.

\subsection{Definitions and notation}\label{subsec:def&not}

In this paper, $G=(V,E)$ will always denote a finite connected graph. Its order is denoted by $n.$
A walk in a graph $G$ is a sequence of vertices $(v_1,v_2, \ldots, v_r)$ such that $v_iv_{i+1}\in E$ for every $1\le i \le r-1.$ It is a closed walk if $v_1=v_r.$ A path is a walk in which all vertices are distinct.  
For a subset $U \subset V,$ the subgraph of $G$ induced by $U$, is denoted by $G[U].$ Here $G[U]=(U, E \cap \binom{U}{2}).$
If $U=\{v_1, v_2, \ldots, v_k\},$ we also use the notation $G[v_1,v_2,\ldots,v_k]$ for $G[U].$
The distance $d(u,v)$ between two vertices $u$ and $v$ equals the length of a shortest path containing them.
The latter is always finite and bounded by $n-1$, as the graph is connected.
The total distance or Wiener index $W(G)$ is the sum of all distances.
The average distance is the arithmetic average of the distances and equals $\mu(G)=\frac{W(G)}{\binom n2}.$

A walk in a digraph $D=(V,A)$ is defined similarly, it is a sequence of vertices $(v_1,v_2, \ldots, v_r)$ such that $\vc{v_iv_{i+1}}\in A$ for every $1\le i \le r-1.$
It is a closed walk if $v_1=v_r,$ and it is a (directed) path if all vertices are distinct.
A digraph is biconnected if for every choice of $u,v \in V$, there is a path from $u$ to $v$ (and vice versa); equivalently, its total distance is finite.
Another term for biconnected is strongly connected.

The Steiner distance $d_k(v_1,v_2, \ldots, v_k)$ of a $k$-tuple of vertices $(v_1,\ldots, v_k)$ in a graph $G$, is defined as the size of a minimum subtree that contains these $k$  vertices.

For a set of $k$ vertices, a Steiner tree in a graph $G$ is a subtree of $G$ of minimum order which contains the $k$ vertices.
A Steiner tree for the set of all vertices of a graph is a spanning tree, so the Steiner tree problem is a generalization of the spanning tree problem.
In $1989$, Chartrand et al.~\cite{COT89} defined the Steiner distance $d_k(v_1,v_2, \ldots, v_k)$ of a $k$-tuple of vertices $(v_1,\ldots, v_k)$ as the size of a minimum connected subgraph (subtree) that contains all vertices. This led to the notion of the average Steiner distance, defined in~\cite{DOS96}, and the Steiner Wiener index, defined in~\cite{LMG16}.
Surveys on Steiner distance and related Steiner distance parameters can be found in~\cite{Oellermannsurvey,MaoSurvey}.
The concept of a Steiner tree extends to digraphs, allowing for the definition of Steiner $k$-distance. The latter can be defined for a $k$-tuple $(v_1, v_2, \ldots, v_k)$ as the minimum size of a Steiner tree connecting the root $v_1$ with all terminal vertices $v_2, \ldots, v_{k-1}$, and $v_k$.
Algorithms to determine the smallest Steiner tree in this setting have been investigated as well; see, e.g.~\cite{Guo90}.
For our purpose, we prefer to define the Steiner distance $d_k(v_1,v_2, \ldots, v_k)$ in a digraph as the minimum number of arcs in a biconnected (strongly connected) subdigraph containing the vertices $(v_i)_{1 \le i \le k},$ which is the analogue of the definition used by Chartrand et al.~\cite{COT89}.
The average Steiner distance is just the average over all $k$-tuples, and the Steiner Wiener index $W_k(D)$ of a digraph is the sum over all $k$-tuples of vertices.
Variants of both the TSP and Steiner tree problem are often NP-hard, but are somewhat easier for restricted graph classes, see, e.g.~\cite{MPP18}.

The $k$-Steiner Wiener index of a graph $G$ equals $\sum_{\binom Vk} d_k(v_1,v_2, \ldots, v_k)$ and the average Steiner distance for $k$-tuples is $\mu_k(G)= \frac{W_k(G)}{ \binom{n}k}.$

The analogue of~\cref{def:Wtspk} for digraphs is the following.
\begin{definition}\label{def:Wtspk_Di}
    In a digraph $D$, we define $\tsp_k(v_1,v_2, \ldots, v_k)$ as the length of a shortest closed walk in $D$ that visits all vertices $v_1,v_2, \ldots, v_k$.\\
    The $k$-TSP Wiener index of a digraph $D$ is defined as 
    $$\Wtspk(D)= \sum_{\binom{V}{k}} \tsp_k(v_1,v_2, \ldots, v_k).$$\\
    The average $k$-TSP distance equals 
    $$\mutspk(D)=\frac{\Wtspk(D)}{\binom{n}{k}}.$$
\end{definition}

For two functions $f$ and $g$, we write $f(x)=o(g(x))$ if $\lim_{ x \to \infty} \frac{ f(x)}{g(x)}=0$. This can also be written as $g(x) \gg f(x).$
We use $f \sim g$ if $\lim_{ x \to \infty} \frac{ f(x)}{g(x)}=1$, or equivalently $f(x)=(1-o(1))g(x).$

\section{Relations between the Steiner Wiener and TSP Wiener index}\label{sec:rel_Wk_Wtspk}

In this section, we prove Theorem~\ref{thr:TSPW=2W}, hereby starting with the inequality itself.

\begin{proposition}\label{prop:TSPWle2W}
    For a graph $G$ and integer $k\ge 2$, we have $\Wtspk(G)\le 2\W_k(G).$ 
\end{proposition}

\begin{proof}
    Given a minimum Steiner tree spanning $k$ vertices $v_1, v_2, \ldots, v_k,$ one can make a closed walk along this Steiner tree visiting every edge precisely twice. The latter follows from induction on the size of a tree and taking a cut-vertex in the induction step.
    This implies that $\tsp_k(v_1, v_2, \ldots, v_k) \le 2d_k(v_1, v_2, \ldots, v_k).$ Summing over $\binom{V}{k},$ we obtain the result.
\end{proof}

The characterization of the graphs that attain equality is harder, except for the case $k=2$. When $k=2,$ a shortest closed walk visiting $v_1$ and $v_2$ will always have length $2d(v_1,v_2).$
Before going on with the characterization for $k \ge 3$, we prove the following theorem relating the $k$-TSP Wiener index with the Wiener index.

\begin{theorem}\label{thr:wtspk_vs_W}
    For every graph $G$ of order $n \ge k$, we have 
    $\Wtspk(G) \le k \binom{n}{k}\mu(G)= \frac{2}{k-1} \binom{n-2}{k-2} W(G).$
    Equality holds if and only if $k \in \{2,3\}$, or $k>3$ and $G$ equals $S_n$ or $K_n$.
\end{theorem}

\begin{proof}
    Let $U=\{v_1, v_2 , \ldots, v_k\}$ be a set of $k$ vertices.
    Once the order $(v_1,v_2, \ldots, v_k)$ of vertices visited in the closed walk is determined, we have 
    $\tsp_k(v_1,v_2, \ldots, v_k)=d(v_1,v_2)+d(v_2,v_3)+\ldots+d(v_{k-1},v_k)+d(v_k,v_1).$
    Let $f(U)=\Exp \sum_{i=1}^k d(v_i,v_{i+1})$ be the average over all $k!$ possible permutations of the $k$ vertices (where we take $v_{k+1}=v_1$ in the sum).
    Then $\tsp_k(v_1,v_2, \ldots, v_k) \le \Exp \sum_{i=1}^k d(v_i,v_{i+1})=f(U)$
    and thus 
    $$
        \Wtspk(G)\le \sum_{U \in \binom{V}{k}}  f(U)= k \binom{n}{k}\mu(G).
    $$
    For $k \in \{2,3\},$ every permutation results in the same sum and so the equality is true for every graph.
    
    For $K_n$, we trivially have $\Wtspk(K_n)=k\binom{n}{k}$ and for a star one easily verifies that
    $$\Wtspk(S_n)=2k \binom{n-1}{k}+ 2(k-1) \binom{n-1}{k-1} = \frac{2k(n-1)}{n}\binom{n}{k}=k \binom{n}{k}\mu(S_n).$$
    
    Now assume that $n \ge k \ge 4$ and $G$ is not a star nor a clique. This implies that $G$ has a spanning tree which is not a star and in particular that $G$ contains two disjoint edges (this is also a consequence of Erd\H{o}s-Ko-Rado results).
    
    Equality did only occur if all permutations gave the same sum, so in particular for every $4$ vertices $v_1,v_2,v_3,v_4$ we have $d(v_1,v_2)+d(v_2,v_3)+d(v_3,v_4)=d(v_1,v_3)+d(v_2,v_3)+d(v_2,v_4),$ which is equivalent with $d(v_1,v_2)+d(v_3,v_4)=d(v_1,v_3)+d(v_2,v_4).$
    By the same argument, the latter also equals $d(v_1,v_4)+d(v_2,v_3).$
    So any two disjoint edges span a $K_4$.
    In particular $G$ contains a $K_4$. Let $K_m$ ($m \ge 4$) be a largest clique in $G$, and let $(v_i)_{1 \le i \le m}$ be its vertices.
    Since $G$ is not a clique itself and connected, there is a vertex $u$ adjacent to some vertices of the clique, but not all of them.
    Withous loss of generality, we assume that $u$ is adjacent to $v_1$, but not adjacent to $v_3$. That is, $uv_1 \in E(G)$ and $uv_3 \not \in E(G).$
    Since $uv_1$ and $v_2v_3$ are disjoint edges in $G$, they span a $K_4$ (by the earlier observation) and so $uv_3 \in E(G)$, which is a contradiction with our assumptions. This contradiction implies that no other graph than $S_n$ and $K_n$ can attain equality.
\end{proof}

    As such, for $k=3,$ we have $\Wtsp3(G)=(n-2)W(G)$ for every graph $G$.
    The characterization of the graphs for which $2W_3(G)=\Wtsp3(G)$ can now be formulated and proven as follows. 
    (The subsequent proof also contains an alternative proof for the inequality $2W_3(G)\geq \Wtsp3(G)$ which the reader might skip.)
    
\begin{theorem}
    Every graph $G$ satisfies $W_3(G) \ge \frac{n-2}{2} W(G).$
    Equality holds if and only if $G$ contains no vertices $u,v,w$ for which $2\max\{d(u,v),d(u,w),d(v,w)\}<d(u,v)+d(u,w)+d(v,w)$ and every choice of $3$ shortest paths between the $3$ pairs $(u,v),(v,w), (w,u)$ are edge-disjoint.
\end{theorem}
\begin{proof}
    By definition, $d_3(u,v,w)$ is the number of edges in a minimum tree $T$ spanning $u,v$ and $w$.
    Either this tree is path, in which case $d_3(u,v,w)=\frac 12 \left(d(u,v)+d(v,w)+d(u,w)\right),$ or it has a single vertex of degree $3$, call it $z$.
    In the latter case, we have $d(u,v) \le d(u,z)+d(z,v)$ by the triangle inequality and the two analogues of this.
    Hence \begin{equation}\label{eq:triangle-ineq}
        d_3(u,v,w)=d(u,z)+d(v,z)+d(w,z) \ge \frac 12 \left(d(u,v)+d(v,w)+d(u,w)\right).
    \end{equation}
    So the latter inequality is true for all triples.
    Summing over all triples in $\binom{V}{3},$ every distance $d(u,v)$ appears $n-2$ times, from which we conclude that $W_3(G) \ge \frac{n-2}{2} W(G).$
    
    If Equality holds in Inequality~\eqref{eq:triangle-ineq} for every triple of vertices $(u,v,w)$, either their Steiner tree is a path and thus $2\max\{d(u,v),d(u,w),d(v,w)\}=d(u,v)+d(u,w)+d(v,w)$, or the $3$ shortest paths between the $3$ pairs $(u,v),(v,w), (w,u)$ can be chosen such that each of them contains $2$ of the $3$ branches of the tree and as such are not edge-disjoint.
    
    Now we prove the other direction.
    Let $(u,v,w)$ be a triple of vertices in $G$ for which inequality~\eqref{eq:triangle-ineq} is strict.
    We can always choose the shortest paths between the pairs in $\{u,v,w\}$ such that their union forms a pseudotree (unicyclic graph or tree). Choose these in such a way that the cycle is as small as possible (not a tree by choice of $(u,v,w)$). This is sketched in Figure~\ref{fig:shortestpaths}.
    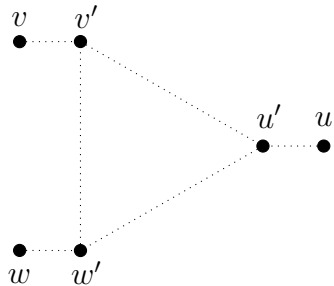
\begin{figure}[h]
    \centering
       \begin{tikzpicture}[x=0.8cm,y=0.8 cm]
    {
	\foreach \a in {0,120,240}{
	\draw[dotted] (\a+120:2) -- (\a:2);
	\draw[fill] (\a:2) circle (0.1);}
    }
    \draw[dotted] (0:3) -- (0:2);
    \draw[dotted] (-2,1.73) -- (120:2);
    \draw[dotted] (-2,-1.73) -- (240:2);
    \draw[fill] (0:3) circle (0.1);
    \draw[fill] (-2,1.73) circle (0.1);
    \draw[fill] (-2,-1.73) circle (0.1);
    \coordinate [label=center:\large \textbf{$u$}] (A) at (3,0.45);
    \coordinate [label=center:\large \textbf{$u'$}] (A) at (2.1,0.5);

    \coordinate [label=center:\large \textbf{$v'$}] (A) at (-0.9,2.15);
    \coordinate [label=center:\large \textbf{$w'$}] (A) at (-0.9,-2.1);
    \coordinate [label=center:\large \textbf{$v$}] (B) at (-2,2.1);
    \coordinate [label=center:\large \textbf{$w$}] (C) at (-2,-2.15);
	\end{tikzpicture}
    \caption{Sketch of shortest paths between the three vertices $(u,v,w)$ }
    \label{fig:shortestpaths}
\end{figure}

    Let $(u',v',w')$ be the three vertices with degree $3$ in this unicyclic graph.
    If there would exist a choice of shortest paths between these $3$ vertices that are not internally vertex-disjoint/ edge-disjoint, this would imply that we had a choice with a smaller cycle, which would be a contradiction. So they are edge-disjoint.
    Also $2\max\{d(u',v'),d(u',w'),d(v',w')\}<d(u',v')+d(u',w')+d(v',w')$ due to the triangle-inequality and the assumption that the shortest paths between them were edge-disjoint.
    \end{proof}

\begin{remark}
 Examples of graphs for which the equality $W_3(G) = \frac{n-2}{2} W(G)$ does hold are graphs all of whose induced cycles are $C_4$s (e.g. complete bipartite graphs) and convex subsets of the $n$-dimensional gridgraph.
 Non-bipartite graphs (consider a smallest odd cycle) and bipartite graphs containing an induced $C_6$ (consider the $3$ vertices of the $C_6$ in the same bipartition class) are examples of graphs for which no Equality holds.
\end{remark}

Finally, we prove the characterization of the extremal graphs in Theorem~\ref{thr:TSPW=2W} for $k \ge 4.$

\begin{proposition}\label{thr:TSPW=2W_tree}
    For a tree $T$ and integer $k\ge 4$, we have $\Wtspk(T)=2\W_k(T).$
    If $G$ is a graph that is not a tree, then $\Wtspk(G)<2\W_k(G).$
\end{proposition}

\begin{proof}
    Since the vertices visited in a shortest closed walk on a tree, form a tree itself and every edge is used exactly twice, we have $\tsp_k(v_1,v_2, \ldots, v_k)=2d_k(v_1,v_2, \ldots, v_k)$ for every $k$-tuple, by summing over $\binom{V}k$ we conclude that $\Wtspk(T)=2\W_k(T).$
    
    Now assume $G$ is not a tree and it has girth $m$, i.e. there is a smallest cycle $C_m$ belonging to $G$.
    If $k \le m,$ one can distribute the points over the cycle in such a way that they do not belong to an interval of length smaller than $\frac m2$ (here we use $k \ge 4$).
    Since the cycle is chosen minimal, this also implies that a minimum Steiner tree will contain at least $\frac {m+1}{2}$ edges. Hence $2d_k(v_1,v_2, \ldots, v_k)\ge m+1.$
    On the other hand one can walk along $C_m$, i.e. put the vertices in the order of the cycle, and thus $\tsp_k(v_1,v_2, \ldots, v_k)\le m.$
    
    If $k>m$, one can take a set of $k$ vertices $v_1$ up to $v_k$, containing the $m$ vertices on the cycle $C_m$ and $k-m$ others (iteratively add a neighbour) in such a way that there is a spanning unicyclic subgraph $H$ in $G[v_1, v_2, \ldots, v_k]$ containing $C_m.$
    We note that $d_k(v_1,v_2, \ldots, v_k)\ge k-1$ in general and in this case, equality holds.    
    Since there is a closed walk, traversing every edge of $H$ twice, except from the edges of $C_m$ which are traversed only once, we have $$\tsp_k(v_1,v_2, \ldots, v_k)\le 2k-m\le 2k-3<2(k-1)=2d_k(v_1,v_2, \ldots, v_k)$$
    and the inequality here is strict, implying $\Wtspk(G)<2\W_k(G).$
\end{proof}

We also prove a relation between the Steiner Wiener and TSP Wiener index for digraphs.

\begin{proposition}\label{prop:TSPWle2W}
    For a digraph $D$ and integer $k\ge 2$, we have $\Wtspk(D)\ge \W_k(D).$ 
\end{proposition}

\begin{proof}
    Let $\{v_1,\ldots,v_k\}$ be a set with $k$ vertices. 
    Take a closed walk in $D$ visiting these $k$ vertices. This results in a biconnected digraph, i.e. the vertices and arcs visited in the closed walk form a biconnected digraph.
    This implies that $d_k(v_1,\ldots,v_k) \le \tsp_k(v_1,\ldots,v_k).$ Summing over all possible $k$-vertex sets in $\binom{V}{k}$, one concludes that $\Wtspk(D)\ge \W_k(D).$ 
\end{proof}

In contrast to Theorem~\ref{thr:TSPW=2W}, the relation between $\Wtspk(D)$ and $\W_k(D)$ already depends on the digraph when $k=2.$
For this, consider the digraph $DP_{n,d}$, which is the directed cycle $\vc{C_d}$ with a blow-up of one vertex by an independent set $\{\ell_1,\ell_2,\ldots,\ell_{n-d+1}\}$ of size $n -d + 1$, presented in Figure~\ref{fig:DPnd} (or~\cite[Fig.3]{SC19}).
We note that the shortest closed walk that visits both $\ell_1$ and $\ell_{2}$ will have length $2d$, while the minimum number of arcs in a biconnected digraph containing both $\ell_1$ and $\ell_{2}$ is $d+2.$ 
For $k$ of these vertices $\ell_i$, we note that $d_k(\ell_1,\ell_2, \ldots, \ell_k) = d-2+2k$ and $\tsp_k(\ell_1,\ell_2, \ldots, \ell_k)=dk$.
One can observe that when $k=o(d)$ (that is, $d\gg k$), the ratio $\frac{ \tsp_k(\ell_1,\ell_2, \ldots, \ell_k) }{d_k(\ell_1,\ell_2, \ldots, \ell_k)} =\frac{dk}{ d+2k-2} \sim k$.
When also $d=o(n)$, most ($100(1-o(1))\ \%$) of the $k$-tuples will only contain vertices of the independent set of size $n-d+1$, since $\frac{\binom{n}{k}} {\binom{n-d+1}{k}} =1-o(1)$, and hence
$\Wtspk(D)\sim k \W_k(D).$ 
On the other hand, $\Wtspk(D)\le k \W_k(D)$ is a trivial upper bound since 
$d(v_1, \ldots, v_k) \ge \max\{d(v_i,v_j) \mid 1 \le i,j \le k\}+1$ and $\tsp_k(v_1, \ldots, v_k) \le k \cdot \max\{d(v_i,v_j) \mid 1 \le i,j \le k\}.$

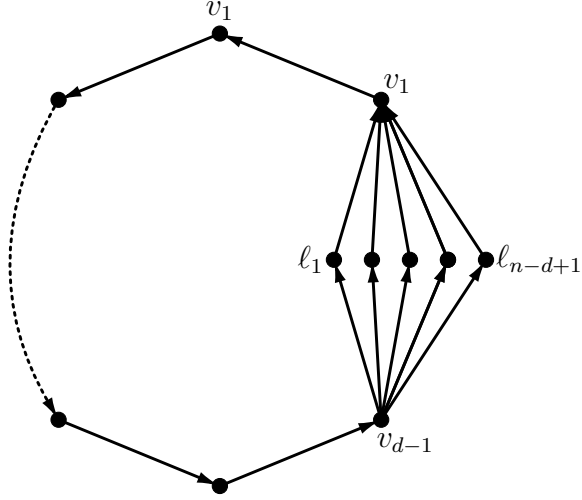
\begin{figure}[h]
	\begin{center}
			\begin{tikzpicture}[line cap=round,line join=round,>= {Latex[length=3mm, width=1.5mm]},x=1cm, y=1 cm]
			{
				\foreach \a in {0,45,90,225,270,315}{
				\draw[fill] (\a:3) circle (0.1);
			    \draw[->,line width=1.1pt] (\a:3) -- (\a+45:3);
			    }
				\foreach \x in {1.5,2,2.5,3,3.5}{
				\draw[->,line width=1.1pt] (0:\x) -- (45:3);
				\draw[->,line width=1.1pt] (-45:3)--(0:\x);
				\draw[fill] (0:\x) circle (0.1);
				}
				\draw [->,line width=1.1pt,dotted] (135:3) to [out=-120, in=120] (225:3)   ;
				\draw[fill] (135:3) circle (0.1);
				
				\coordinate [label=left:\large \textbf{$\ell_1$}] (A) at (0:1.5);
				\coordinate [label=right:\large \textbf{$\ell_{n-d+1}$}] (A) at (0:3.5);
					\coordinate [label=center:\large \textbf{$v_1$}] (A) at (45:3.3);
						\coordinate [label=center:\large \textbf{$v_2$}] (A) at (90:3.3);
                \coordinate [label=center:\large \textbf{$v_3$}] (A) at (135:3.4);
					\coordinate [label=center:\large \textbf{$v_{d-1}$}] (A) at (315:3.45);
			}
			\end{tikzpicture}
		\end{center}
    \caption{The digraph $DP_{n,d}$}
    \label{fig:DPnd}
\end{figure}

\section{Bounds on the average TSP distance}\label{sec:bounds_mutspk}

\begin{theorem}
    For every $k \ge 2$ and graph $G$ of order $n$, $$k \le \mutspk(G) \le 2\frac{k-1}{k+1}(n+1),$$
    or equivalently
    $$\Wtspk(K_n) \le \Wtspk(G) \le \Wtspk(P_n).$$
    Equality for the upper bound does hold if and only if $G$ is the path $P_n$, or $k=n$ and $G$ is a tree.
    The lower bound is true if and only if every $k$ vertices in $G$ span a $C_k$, i.e. form a Hamiltonian subset.
\end{theorem}

\begin{proof}
    The lower bound is immediate, since $\tsp_k(v_1,v_2,\ldots,v_k) \ge k$ for every $k$-set of vertices.
    Equality holds if and only if $G[v_1,v_2,\ldots,v_k]$ contains a spanning cycle, i.e. $G[v_1,v_2,\ldots,v_k]$ is Hamiltonian, for every $k$-set $\{v_1,v_2,\ldots,v_k\} \subset V.$
    
    The upper bound is true by combining Theorem~\ref{thr:TSPW=2W} and~\cite[Thr.~3]{LMG16}, as together they imply that $\Wtspk(G) \le 2 W_k(G) \le 2 W_k(P_n)=\Wtspk(P_n).$
    Equality holds if and only if $\Wtspk(G) = 2W_k(G)$ and $\mu_k(G)=\frac{k-1}{k+1}(n+1)$ which is when $G=P_n$ or $k=n$ and $G$ is a tree, by combining Theorem~\ref{thr:TSPW=2W} and~\cite[Thr.~2.1]{DOS96}. Here we note that if $k=n=2,$ $P_2$ is the only graph on $2$ vertices and when $k=3$ $\mutsp3(K_3)<\mutsp3(P_3).$
\end{proof}

\begin{remark}
    For $k \in \{2,3\},$ the clique $K_n$ is the unique graph attaining the minimum. This is also a corollary of the result for the average distance and Theorem~\ref{thr:wtspk_vs_W}.
    For $k \ge 4,$ the condition $\delta(G) \ge n+2-k$ is necessary to ensure that $\delta(G[v_1,v_2,\ldots,v_k])\ge 2$ for every $k$-tuple, but this is only sufficient when $k \le 5.$
\end{remark}

\begin{remark}
    The analogue of~\cite[Thr.~2.2]{DOS96} does also hold for the average TSP distance for both graphs and digraphs. 
    That is, if $D$ be a biconnected, weighted digraph and $2 \le j, k$ be integers with $j+k-1 \le n.$ Then $\mu_{tsp,j+k-1}(D) \le  \mutspk(D) + \mu_{tsp,j}(D).$
    The underlying reason being that the union of two closed walks with a vertex in common can be combined into a closed walk whose (weighted) length is the sum of the other two.
\end{remark}

Knowing the extremal bounds for graphs of order $n$, without any constraint, it is natural to wonder about the extremal values and graphs for certain logical constraints.
One of these is the diameter $d$ of the graph.
Plesn\'{\i}k~\cite{P84} asked these questions for the average distance, and it turned out that determining the maximum average distance given order and diameter is a very hard question. The best estimations known at this point can be found in~\cite{SC19}.
One beautiful conjecture related to the problem of Plesn\'{\i}k is the one by DeLaVi\~{n}a and Waller~\cite{DLVW}. Here they conjecture that when the order equals $2d+1$ and the diameter $d$ is at least $3$, then the cycle $C_{2d+1}$ maximizes the average distance.
As a corollary of Theorem~\ref{thr:wtspk_vs_W}, we observe that the DeLaVi\~{n}a-Waller conjecture is true for the TSP Wiener index of order $k$ for $k \in \{2,3\}$ if and only if the original version is true.
In the remaining of this section, we will show that the analogue would not go through for the TSP variant of the problem when $k \ge 4.$ The essence is that $\mutspk(C_n) \le n$, while for a tree with order $n$ we may have $\mutspk(T) \sim 2n$ when $k$ is reasonably large.

Before proving this, we estimate $\mutspk(C_n).$

\begin{lemma}\label{lem:Cn_mutspk}
        For $k$ fixed and $n$ sufficiently large, we have 
        $$\mutspk(C_n) \sim \left(1-\frac{1}{2^{k-1}} \right) n.$$
\end{lemma}
    \begin{proof}
        Let $d=\floorfrac{n}{2}.$
        Note that if $k$ vertices $v_1\ldots v_k$ do not belong to an interval of length at most $d$, then $\tsp_k(v_1,\ldots,v_k)=n.$
        If they do belong to an interval of length $i \le d$, then $\tsp_k(v_1,\ldots,v_k)=2i.$
        
     As such, using $\binom{n}{k}=\frac{n}{k}\binom{n-1}{k-1}$ and $\binom{n}{k}\sim \frac{n^k}{k!} \sim 2^k \binom{d}{k}$, we can compute that 
     \begin{align*}
         \Wtspk(C_n)&= n \sum_{i=k}^d \binom{i-2}{k-2}\cdot 2\cdot (i-1) + \left( \binom{n}{k} - n\sum_{i=k}^d \binom{i-2}{k-2}\right)n\\
         &= 2n(k-1) \sum_{i=k}^d \binom{i-1}{k-1} + \left( \binom{n}{k} - n \binom{d-1}{k-1}\right)n\\
         &= 2n(k-1) \binom{d}{k} + \left( \binom{n}{k} - n \binom{d-1}{k-1}\right)n\\
         &\sim 2n(k-1) \frac{1}{2^k} \binom{n}{k} + n\binom{n}{k} -nk \frac{1}{2^{k-1}} \binom{n}{k}\\
         &=  \binom{n}{k} \left( 1 - \frac{1}{2^{k-1}}\right)n
     \end{align*}
     
     Here we used that $\sum_{i=a}^b \binom{i}{a} = \binom{b+1}{a+1}$, which can be proven by induction.     
    \end{proof}

In the next theorem, we prove that the tree presented in Figure~\ref{fig:tree_beatingC_2d+1_wtspk} has a larger $k$-TSP Wiener index for any $k \ge 4,$ when $n$ (equivalently $d$) is sufficiently large.

\begin{figure}[h]
	\begin{center}
			\begin{tikzpicture}[x=1.1cm, y=1.1 cm]
			{
				\foreach \a in {0,120,240}{
				\draw[fill] (\a:3) circle (0.1);
				\draw[fill] (\a:2.5) circle (0.1);
				\draw[fill] (\a:2) circle (0.1);
				\draw[fill] (\a:0.5) circle (0.1);
				\foreach \x in {-10,-5,5,10}{
				\draw[thick] (\a+\x:3.5) -- (\a:3);
				\draw[dotted] (\a+5:3.5) -- (\a-5:3.5);
				\draw[fill] (\a+\x:3.5) circle (0.1);}		\draw[dotted] (\a:0.5) -- (\a:2);
				\draw[thick] (0:0) -- (\a:0.5);
				\draw[thick] (\a:2) -- (\a:3);
				}
				\draw [decorate,decoration={brace,amplitude=4pt},xshift=0pt,yshift=0pt] (10:3.7)--(-10:3.7)  node [black,midway,xshift=0.35cm]{$\frac d6$};	
				
				\coordinate [label=center:\large \textbf{$\frac d2$}] (A) at (1.5,0.3);
				\draw[fill] (0:0) circle (0.1);
			}
			\end{tikzpicture}
		\end{center}
    \caption{A symmetric tree with order $2d+1$ and diameter $d$ with large $\mutspk(T)$}
    \label{fig:tree_beatingC_2d+1_wtspk}
\end{figure}
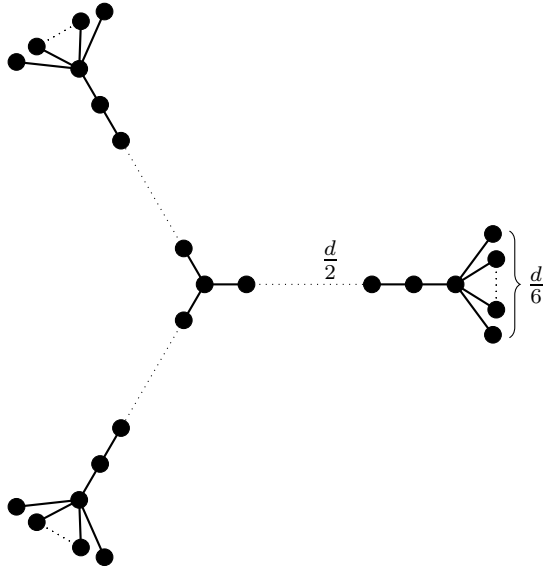

\begin{theorem}
    When $k \ge 4$ and $d$ is sufficiently large (in terms of $k$), there exist trees with diameter $d$ and order $n=2d+1$ for which $\mutspk(C_{n})<\mutspk(T).$
\end{theorem}

\begin{proof}
    Take the tree $T$ which consists of a central vertex of degree $3$ which is connected to $3$ brooms of length $\left \lfloor \frac d2 \right \rceil $ and $\left \lfloor \frac d6 \right \rceil$ leaves each, in such a way that the number of leaves of $T$ is $\floorfrac d2$, the order of $T$ is $2d+1$ and its diameter equals $d.$ 
    The shortest closed walk containing $k$ vertices, will contain an edge $e=uv$ twice if and only if among the $k$ vertices, there are vertices closer to $u$ than to $v$ and also in the reverse order.
    If $e$ is an edge such that $T \backslash e$ has two components with $a$ resp. $b$ vertices ($a+b=n$), then the probability that the $k$ vertices are in the same component is $\frac{\binom{a}{k}+\binom{b}{k}}{\binom{n}{k}}.$
    When $a$ and $b$ are both much larger than $k$, this is approximately $\frac{a^k+b^k}{(a+b)^k}$, i.e., of the form $(1+o(1))\frac{a^k+b^k}{(a+b)^k}.$
    The latter follows since $\binom{a}{k}\sim \frac{a^k}{k!}$ for $k=o(a)$ and similar for $b$.
    For $d$ sufficiently large in terms of $k$, the latter approximation is valid for every non-pendent edge (edge not incident with a leaf).
    The contribution of the pendent edges to $\mutspk(T)$ is $$2\floorfrac{d}{2} \frac{\binom{n}{k}-\binom{n-1}{k}}{\binom{n}{k}}= 2\floorfrac{d}{2} \frac{\binom{n-1}{k-1}}{\binom{n}{k}}=2\floorfrac{d}{2} \frac{k}{n}\sim \frac{k}{2}=o(d).$$

    Remembering that every edge in the closed walk has to be traversed twice, the expected value for $\mutspk(T)$ is 
    $$ 2\floorfrac{d}{2} \frac{\binom{n-1}{k-1}}{\binom{n}{k}}+ 6 \cdot \sum_{a=\left \lfloor \frac d6 \right \rceil}^{\left \lfloor \frac {2d}{3} \right \rceil}  \frac{\binom{a}{k}+\binom{n-a}{k}}{\binom{n}{k}}.$$
    
    By the previously explained approximations, this is well-estimated by 
 \begin{align*}
    6 \cdot \sum_{a=\left \lfloor \frac d6 \right \rceil}^{\left \lfloor \frac {2d}{3} \right \rceil}\left(1- \frac{ a^k + (2d-a)^k}{(2d)^k} \right) & \sim 6\cdot \int_{ a=\frac d{6}}^{\frac{2d}3} \left(1- \frac{ a^k + (2d-a)^k}{(2d)^k} \right) \d a\\
     & \sim 6\cdot \int_{ \frac 1{12}}^{\frac13} \left( 1-x^k-(1-x)^k \right) \d x \cdot 2d
 \end{align*}
    When $k \ge 6,$ this is larger than $2d+1=n>\mutspk(C_{n})$ and so the result follows as we assume $d$ is sufficiently large.
    For $k=5$, it is larger than $1.98d$ and for $k=4$ it is larger than $1.752 d$.
    On the other hand, by Lemma~\ref{lem:Cn_mutspk} we have $\mutspk(C_n) \sim \left(2-\frac{1}{2^{k-2}} \right) d,$ which is smaller than the two values for the tree when $k \in \{4,5\}.$
\end{proof}


\section{Conclusion and further directions}\label{sec:conc}

In this paper, we defined the average TSP distance of a graph and digraph, as well as a version for the Steiner distance for digraphs.
We compared and investigated the extremal behaviour of average TSP distance and average Steiner distance.
While solving a general TSP instance precisely is a hard problem, the study of the average TSP distance, e.g. for random graphs or certain graph classes, may give some fundamental insights on the travelling salesman problem. 
Also certain problems for other distance measures might be interesting.
Let $n$ and $k$ be integers such that $2 \le k \le n$ and let $G$ be a graph of order $n.$
The TSP-$k$-eccentricity of a vertex $v$ of $G$ equals $\ecc_{tsp,k}(G,v)=\ecc_{tsp,k}(v)= \max\{\tsp_k(S) \mid S\subset V, \abs{S}=k,  v \in S\}.$
The TSP-$k$-radius and TSP-$k$-diameter of a vertex $v$ of $G$ are defined by $\rad_{tsp,k}(G)=\min\{\ecc_{tsp,k}(v) \mid v \in V\}$ and $\diam_{tsp,k}(G)=\max\{\ecc_{tsp,k}(v) \mid v \in V\}$.
The TSP-$k$-radius is related to the worst solution of a TSP instance with $k-1$ destinations and an optimally chosen base location.

A few observations are immediate for the eccentricities (and the same is true for the radius and diameter analogues).
For a tree $T$, $\ecc_{tsp,k}$ equals $2\ecc_k$.
If $H$ is a spanning subgraph of $G$, then $\ecc_{tsp,k}(G,v)\le \ecc_{tsp,k}(H,v)$.
By starting from a clique and deleting one edge belonging to a triangle at a time till one ends with a path, for $k \ge 3$ one can observe that given the order, $\ecc_{tsp,k}(v)$ can take any value between $k$ and $2(n-1)$. 
It would be interesting to see more surprising behaviour/ results that gives additional intuition on TSP instances. 
\section*{Acknowledgement}
We thank the two reviewers for careful reading and many suggestions to improve the readability and the exposition of this paper.
We would like to express our gratitude to the American Mathematical Society for organizing the Mathematics Research Community workshop ``Trees in Many Contexts", that triggered this work.
This workshop was supported by the National Science Foundation under Grant Number DMS $1916439$.

\textbf{Open access statement.} For the purpose of open access,
a CC BY public copyright license is applied
to any Author Accepted Manuscript (AAM)
arising from this submission.

\end{document}